\documentclass[12pt]{amsart}

\usepackage{lmodern}

\usepackage{graphicx}
\usepackage{amsmath} 
\usepackage{amssymb}
\usepackage{amsthm}
\usepackage{latexsym}
\usepackage{color}
\usepackage{enumerate}
\usepackage{mathrsfs}
\usepackage{setspace}
\usepackage{tikz}
\usetikzlibrary{matrix}

\newtheorem{definition}{Definition}[section]
\newtheorem{lemma}{Lemma}[section]
\newtheorem{thm}{Theorem}[section]

\newtheorem{prop}{Proposition}[section]

\newtheorem{remark}{Remark}[section]
\numberwithin{equation}{section}

\newcommand{\veps}{\varepsilon}

\def\tr{\textmd{tr}}

\def\Ric{\textmd{Ric}}

\def\R{\mathbb{R}}

\def\tg{\widetilde{g}}

\def\R{\mathbb{R}}

\def\S{\Sigma}
\def\({\left(}
\def\){\right)}

\def\={\stackrel{(n=2)}{=}}
\def\p{\partial}

\def\C{\mathcal{C}}
\def\Gd{\mathcal{G}_\delta}
\def\Hd{\mathcal{H}_\delta}

\newcommand{\be}{\begin{equation}}
\newcommand{\ee}{\end{equation}}
\newcommand{\bee}{\begin{equation*}}
\newcommand{\eee}{\end{equation*}}

\newcommand{\m}{\mathfrak{m}}

\allowdisplaybreaks

\begin{document}

	\title[Gluing Bartnik extensions]{Gluing Bartnik extensions, continuity of the Bartnik mass, and the equivalence of definitions}
\author[McCormick]{Stephen McCormick}
\address{Matematiska institutionen, Uppsala universitet, 751 06 Uppsala, Sweden.}
\email{stephen.mccormick@math.uu.se}
	\maketitle

\begin{abstract}
	 In the context of the Bartnik mass, there are two fundamentally different notions of an extension of some compact Riemannian manifold $(\Omega,\gamma)$ with boundary. In one case, the extension is taken to be a manifold without boundary in which $(\Omega,\gamma)$ embeds isometrically, and in the other case the extension is taken to be a manifold with boundary where the boundary data is determined by $\p\Omega$.
	 
	 We give a type of convexity condition under which we can say both of these types of extensions indeed yield the same value for the Bartnik mass. Under the same hypotheses we prove that the Bartnik mass varies continuously with respect to the boundary data. This also provides a method to use estimates for the Bartnik mass of constant mean curvature (CMC) Bartnik data, to obtain estimates for the Bartnik mass of non-CMC Bartnik data. The key idea for these results is a method for gluing Bartnik extensions of given Bartnik data to other nearby Bartnik data.
\end{abstract}

\section{Introduction}\label{S-intro}
The problem of quasi-local mass in general relativity is the problem of measuring the total mass -- including gravitational energy -- contained in a region of finite extent. There are many proposed definitions of quasi-local mass, each with their own advantages and disadvantages. The definition under consideration in this article is that due to Bartnik \cite{Bartnik-89}, which is generally considered to be one of the most likely definitions to give a true measure of the mass according to physics. The problem with the Bartnik mass is that its abstract definition makes it incredibly difficult to work with directly.

The Bartnik mass is effectively a localisation of the ADM mass, which is the well-known definition of the mass of an asymptotically flat manifold. The reader is reminded of the definition of the ADM mass and that of asymptotic flatness at the end of this section. 

A slightly informal definition of Bartnik's quasi-local mass is given as follows. It is made more precise later by our discussions of what it means to be an `admissible extension'.

\begin{definition}[Informal]
	Let $(\Omega,\gamma)$ be a compact manifold with boundary $\p\Omega$, and nonnegative scalar curvature. The \emph{Bartnik mass} of $(\Omega,\gamma)$ is then defined as
	\be \label{eq-defn-informal}
		\m_B(\Omega,\gamma):=\inf\{\m_{ADM}(M,\widehat{\gamma}):(M,\widehat\gamma)\text{ is an admissible extension of } (\Omega,\gamma) \},
	\ee 
	where $\m_{ADM}$ is the ADM mass of $(M,\widehat\gamma)$. 
\end{definition}

The primary difficulty in computing the Bartnik mass lies in the fact that it is taken to be an infimum over a space of manifolds; an infimum that may or may not be realised. In order to make the definition precise, we also must discuss what constitutes an \emph{admissible extension} in this context.

Intuitively, an admissible extension $(M,\widehat\gamma)$ of $(\Omega,\gamma)$ must satisfy three properties:
\begin{enumerate}[(i)]
	\item $(M,\widehat{\gamma})$ must in some meaningful way, extend $(\Omega,\gamma)$ to an asymptotically flat manifold containing it;
	\item $(M,\widehat{\gamma})$ must satisfy some version of the positive mass theorem, as to not introduce negative mass somewhere in the exterior;
	\item  $(M,\widehat{\gamma})$ must satisfy some kind of non-degeneracy condition to ensure the mass is not always trivially zero.
\end{enumerate}
To illustrate the motivation and necessity of point (iii), consider some $(\Omega,\gamma)$ that isometrically embeds into an asymptotically flat manifold $(M,\widehat{\gamma})$, inside of a horizon of very small area. In this case, $\Omega$ cannot be `seen' by infinity, and could therefore in principle have arbitrarily small mass if the horizon is very small. A common non-degeneracy condition to enforce -- indeed the original one that Bartnik suggested -- is that $(M,\widehat{\gamma})$ cannot contain any minimal surfaces containing $\Omega$.

Another common non-degeneracy condition that is slightly stronger than the above one, is the condition that $\p\Omega$ be outer-minimising in $(M,\widehat{\gamma})$. That is, $\p\Omega$ has the least area among all competing surfaces enclosing it. This condition is particularly useful to impose on the set of admissible extensions, as in this case the Bartnik mass is bound from below by the Hawking mass of $\p\Omega$; a property which follows from Huisken and Ilmanen's proof of the Riemannian Penrose inequality \cite{H-I01}.

There are also variations of the Bartnik mass that involve different definitions in regard to properties (i) and (ii). These differences are what we address in this article. The original definition of Bartnik was that an admissible extension should be an asymptotically flat manifold without boundary, in which $(\Omega,\gamma)$ isometrically embeds and satisfies the non-degeneracy condition. However, there is no reason to expect that if the infimum in \eqref{eq-defn-informal} is realised that the minimiser is smooth. In fact, one expects that such a minimising manifold is in general only Lipschitz continuous along $\p\Omega$. One could then say that condition (i) above, could simply be that an extension $(M,\widehat{\gamma})$ of $(\Omega,\gamma)$ is an asymptotically flat manifold with boundary $\p M$ such that the induced metric on $\p M$ is isometric to $\gamma_{\p\Omega}$, so that $\p M$ can be identified with $\p \Omega$ to obtain a Lipschitz manifold. In this case, the appropriate positive mass theorem for condition (ii) is the positive mass theorem with corners (see Theorem \ref{thm-PMTcorner} below). This says that the mass is positive if, in addition to the usual condition that the scalar curvature is nonnegative, the mean curvature (pointing towards infinity) with respect to $\widehat{\gamma}$ induced on $\p M$ is no larger than the (outward) mean curvature of $\p \Omega$ with respect to $\gamma$. This then leads to the Bartnik mass often being assigned to a triple $(\S,g,H)$, consisting of a closed $2$-manifold $\S$ equipped with a Riemannian metric $g$ and a positive function $H$ (the reader is directed to \cite{Bartnik-TsingHua} for a discussion on this point). If $\S$ is the boundary of some compact Riemannian manifold $(\Omega,\gamma)$, $g$ is the induced metric on $\S$ and $H$ is the outward mean curvature of $\S$, then it is conjectured that the Bartnik mass defined in terms of the boundary geometry is the same as that defined by considering smooth embeddings of $(\Omega,\gamma)$ into asymptotically flat manifolds. In Section \ref{S-defns} we discuss this in more detail, and Theorem \ref{thm-defns} therein gives a condition under which both these definitions do indeed agree.

Under the same hypotheses, two additional related results on the Bartnik mass are given. We prove that the Bartnik mass is continuous with respect to $(g,H)$ in the $C^{2,\tau}\times C^2$ topology, and we show how estimates of the Bartnik mass of constant mean curvature (CMC) surfaces yield estimates of the Bartnik mass in the non-CMC case.

These results follow from the rather simple idea of taking an admissible extension of some $(\Omega,\gamma)$ and gluing in a small collar manifold at the boundary in order to slightly change the boundary geometry. In order to prove the equivalence of the different definitions of Bartnik mass, we simply show that from one kind of admissible extension, we can construct an extension that is admissible in another sense, while controlling the ADM mass.
The idea behind proving continuity is similar.

In Section \ref{S-collar}, we introduce the collars that are used to connect nearby metrics on closed $2$-manifolds. In Section \ref{S-smoothing} we provide tools to locally smooth manifolds with corners, in order to smoothly glue collars to asymptotically flat extensions. Then in Section \ref{S-defns}, we use these tools to give conditions under which three different definitions of Bartnik mass yield the same value. Finally, in Section \ref{S-Cont}, we prove that under these same conditions the Bartnik mass is continuous with respect to the boundary data.

It should be remarked that after completing the first version of this article, the author was made aware that Jauregui had independently established several of these results \cite{J18}. In particular, Jauregui is able to show the equivalence between different notions of extensions without the `convexity condition' required for our $\veps$-collar connections, used throughout.

Before continuing, we briefly take a moment to remind the reader of some definitions. We also remark here that we adhere to a convention of using $g$ for metrics on a surface and $\gamma$ for metrics on a 3-manifold.
\begin{definition}
	A Riemannian $3$-manifold $(M,\gamma)$ is said to be \emph{asymptotically flat}, if there exists a compact set $K$ such that $M\setminus K$ is diffeomorphic to $\R^3$ minus a closed ball, satisfying appropriate decay conditions on the metric. Specifically, in the standard Cartesian co-ordinates coming from such a diffeomorphism, we ask that the metric satisfy
	\bee 
		|\gamma-\delta|+r|\p \gamma|+r^2|\p^2 \gamma|=O(r^{-\tau}),
	\eee 
	where $\delta$ is the usual flat metric and $\tau>\frac12$. Furthermore, we also assume that the scalar curvature $R(\gamma)$ is in $L^1(M)$.
\end{definition}
Usually an asymptotically flat manifold is permitted to have several asymptotic `ends', each diffeomorphic to a copy of $\R^3$ minus a closed ball. However, as we do not need to consider multiple ends here, we omit reference to them for the sake of simplicity. Asymptotic flatness permits us to define a physically meaningful total mass of the manifold, due to Arnowitt, Deser and Miser \cite{ADM}.

\newpage

\begin{definition}
	The \emph{ADM mass} of an asymptotically flat manifold $(M,\gamma)$ may be computed using the standard Cartesian coordinates at infinity by the following expression
	\be \label{eq-ADMdefn}
		\m_{ADM}:=\lim\limits_{r\to\infty}\frac{1}{16\pi}\int_{S_r}(\p_i \gamma_{ij}-\p_j \gamma_{ii})dS^j,
	\ee 
where $S_r$ is a large coordinate sphere and repeated indices are summed over.

It is now well-known that this quantity is purely geometric, and can be defined independently of the coordinates and limiting surfaces used for the computation \cite{Bartnik-86,Chrusciel}.
\end{definition}

\subsection*{Summary of main results}\hspace{30mm}

	A brief summary of the main results is given here for the readers' convenience.\\
	
	Proposition \ref{prop-cornersmooth} localises Miao's corner smoothing technique, so that we can smooth out a manifold with corner without changing its boundary geometry.\\
	
	Under a convexity kind of condition \eqref{eq-convexity-cond}, Theorem \ref{thm-defns} shows that the Bartnik mass defined in terms of isometric embeddings of $3$-manifolds yields the same value as when it is defined in terms Bartnik's geometric boundary data.\\
	
	A consequence of Theorem \ref{thm-defns}, stated as Remark \ref{rmk-CMC}, shows how known estimates for CMC Bartnik data can be used to gives estimates for non-CMC Bartnik data.\\
	
	In Theorem \ref{thm-cont}, we show that under condition \eqref{eq-convexity-cond} the Bartnik mass is continuous with respect to Bartnik data $(\S,g,H)$ in the $C^{2,\tau}\times C^2$ topology.

\vspace{15mm}
	
\section*{Acknowledgements}
This paper was born of a conversation with Pengzi Miao at the program \emph{Geometry and Relativity} hosted by the Erwin Schr\"odinger Institute in 2017. The author would like to thank the ESI for their hospitality and to thank Pengzi Miao for many stimulating conversations on this subject. The author would also like to thank Jeff Jauregui for pointing out an oversight in regard to the non-degeneracy condition, which was present in an earlier version of this article, as well as useful discussions pertaining to this issue.

\newpage

\section{Construction of collars}\label{S-collar}
The central idea here is the gluing together of a some ``interior" manifold with boundary to an asymptotically flat ``exterior" manifold with boundary, whose respective boundaries are somehow close. Specifically, we will ask that the Bartnik data of the boundaries be close in some norm. By \emph{Bartnik data}, we mean a triple $(\S,g,H)$ where $\S$ is a closed 2-surface, $g$ is a Riemannian metric on $\S$ and $H$ is a positive function. When we speak of the Bartnik data of a $3$-manifold (technically the boundary of a $3$-manifold), then $H$ is to be understood as the mean curvature of the boundary. Where this $3$-manifold is being treated as an interior manifold, then we take mean curvature to be with respect to the outward-pointing normal, and when the $3$-manifold is asymptotically flat then mean curvature will be taken with respect to the inward-pointing normal.

 Let $(\Omega,\gamma_1)$ be a Riemannian manifold and let $(\p\Omega,g_1)$ be some connected component of the boundary, and let $(M,\gamma_2)$ be an asymptotically flat manifold with boundary ${(\p M\cong \p\Omega,g_2)}$. We would like to glue these manifolds together to obtain an asymptotically flat manifold that contains a region isometric to $(\Omega,\gamma_1)$ minus a neighbourhood of $\p\Omega$, contains no minimal surfaces outside this region, and has ADM mass close to the ADM mass of $(M,\gamma_2)$. Where this closeness of the ADM mass is controlled by how close the Bartnik data of $\p\Omega$ is to that of $\p M$ in some topology.
 
 This is achieved by rescaling $(M,\gamma_2)$ slightly to ensure that $(\Omega,\gamma_1)$ can ``fit" inside\footnote{See Figure \ref{fig1}, below.} without introducing a minimal surface, then interpolating between the two manifolds with a collar manifold.

The collar is constructed in terms of the Bartnik data on the boundary of both manifolds. Given two nearby metrics $g_1$ and $g_2$ on $\S$, and a positive function $H_1$ on $\S$, we seek to construct a $3$-manifold $\C$ with nonnegative scalar curvature, with two connected components of the boundary. One component inducing the Bartnik data $(\S,g_1,H_1)$ and the other inducing the data $(\S,(1+\epsilon)^2g_2,\widehat{H}_1)$, where $\widehat H_1>\frac{1}{1+\epsilon} H_1$; also such that the boundary component $(\S,g_1,H_1)$ minimises area among all homologous surfaces in $\C$.

Clearly the induced metric on each end of the collar is so that the boundaries of $M$ and $\Omega$ can be identified with a component of the boundary of $\C$. The condition on the mean curvature is so that if the mean curvature $H_2$ of $(\p M,g_2)$ is sufficiently close to $H_1$, then $\C$ can be glued to $M$ while preserving nonnegativity of the scalar curvature. An illustration of how the collar is used is shown below in Figure \ref{fig1}. We remark that the collars we construct here require that $g_1$ and $g_2$ have the same area, so we also need to perform some rescaling; however, we omit reference to this in the diagram below for the sake of exposition.

\begin{figure}[h]
	\includegraphics[width=140mm]{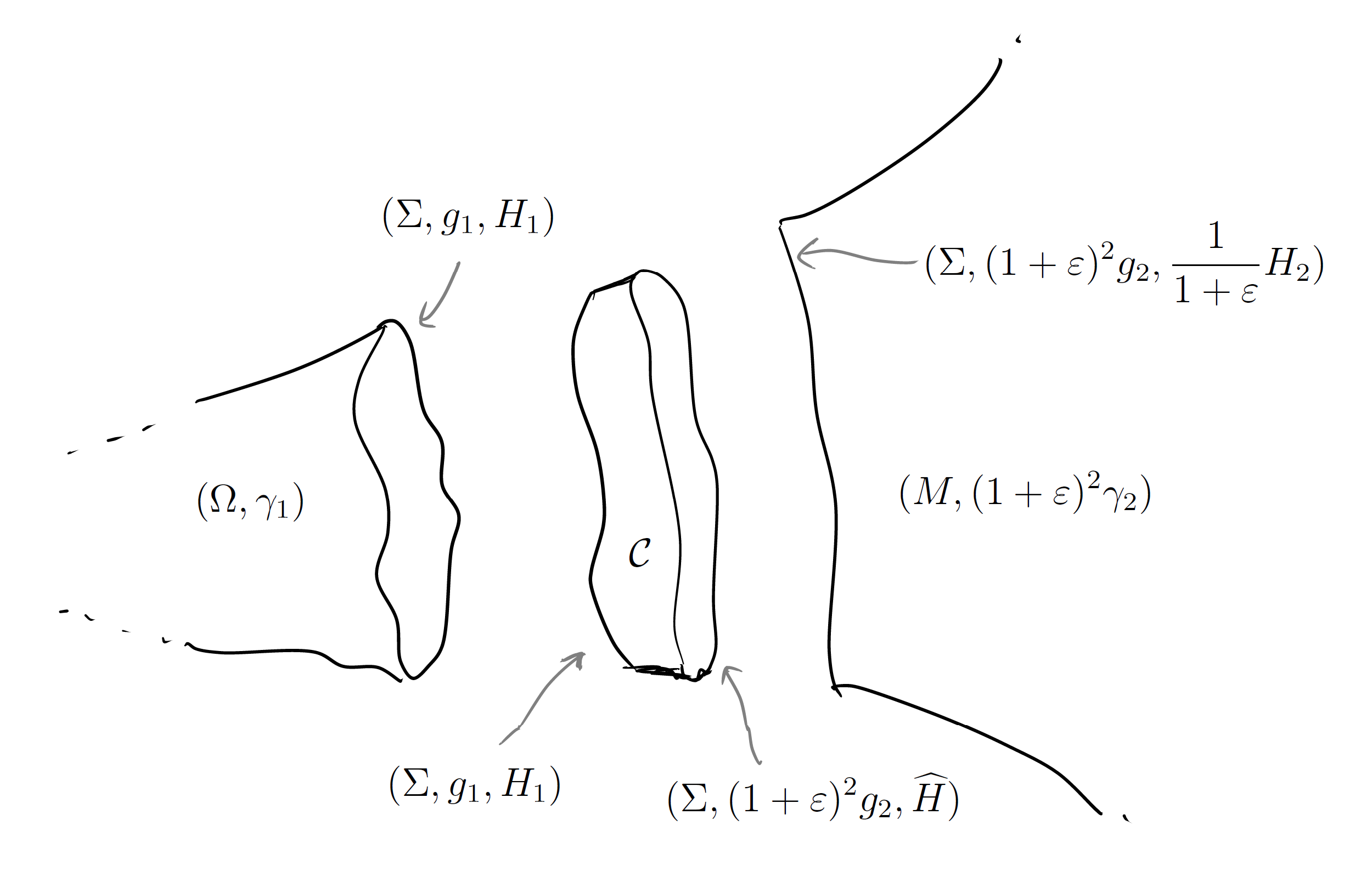}
	\caption{By ensuring $\widehat{H}\geq \frac{1}{1+\veps}H_2$, the collar can be inserted between $\Omega$ and $M$ in such a way that the corners can be smoothed while preserving non-negativity of the scalar curvature. If $g_1$ is close to $g_2$ in some sense, then the collar can be constructed to be small and therefore $\veps$ can also be small.}
	\label{fig1}
\end{figure}

The collar $\C$ used here is very similar to what is used in \cite{MM16}, which was inspired by the work of Mantoulidis and Schoen \cite{M-S}. In fact, these kinds of collars have proven quite useful in relation to the Bartnik mass \cite{CCMM,CCM}, and quasi-local mass quantities more generally \cite{M-X}.

For the sake of presentation, we give a precise definition of the type of collar that we will make use of throughout.

\begin{definition}
	Given Bartnik data $(\Sigma,g_1,H_1)$ and metric $g_2$ on $\Sigma$ satisfying $|\Sigma|_{g_1}=|\Sigma|_{g_2}$, we define an \emph{$\veps$-collar connection} between the Bartnik data and $g_2$, as a manifold $([0,1]\times\Sigma,\gamma)$ satisfying:
	\begin{enumerate}[(i)]
		\item $\gamma$ has nonnegative scalar curvature,
		\item the induced metric and mean curvature on $\Sigma_0$ are $g_1$ and $H_1$ respectively,
		\item the induced metric on $\Sigma_1$ is $(1+\epsilon)^2\,g_2$ and the mean curvature $\widehat H_1$ on $\Sigma_1$ satisfies $H_1>\widehat H_1>\frac{1}{(1+\epsilon)}H_1$,
		\item each leaf $\Sigma_t$ is mean convex.
	\end{enumerate}
\end{definition}

Most of what follows relies on making use of such a collar connection, so we now demonstrate how one can be constructed.

Let $g_1$ and $g_2$ be two metrics on a closed $2$-manifold $\Sigma$, with positive Gaussian curvature. By Lemma 1.2 of \cite{M-S}, there exists a smooth path of metrics $g(s)$ with positive Gaussian curvature satisfying:
\begin{enumerate}
	\item $g(0)=g_1$,
	\item $g(1)=g_2$,
	\item $\tr_{g(s)}g'(s)$.
\end{enumerate}

Following Miao and Xie \cite{M-X}, we define two constants that are associated to a given path:
\be 
\alpha = \frac14 \max_{t}|\dot g(s)|\qquad \beta=\min_{x,t}K(g(s)).
\ee 
The constant $\alpha$ in some sense measures how much $g_1$ and $g_2$ differ, however clearly alpha depends on the chosen path here. We would like to now fix a path in such a way that $\alpha$ can be controlled by $g_1$ and $g_2$. In the case considered by Miao and Xie \cite{M-X}, they demonstrate how such a path can be constructed. In their case, $g_2$ is taken to be round, but if one follows the proof of Proposition 4.1 in \cite{M-X}, it is clear that the proof remains valid if $g_2$ is replaced with any arbitrary metric with positive Gaussian curvature. That is, we have the following lemma.

\begin{lemma}[Proposition 4.1 of \cite{M-X}]\label{lem-alphacontrol}
	The path $g(s)$ as described above can be chosen such that, for any $\epsilon>0$ there is a $\delta>0$ such that if $\|g_1-g_2\|_{C^{2,\alpha}}<\delta$ then $\alpha<\epsilon$.
\end{lemma}
Throughout the remainder of this article, we will take any such path of metrics to be chosen as in Lemma \ref{lem-alphacontrol}. It is worth remarking here that the proof in fact only requires $C^{2,\tau}$ closeness to control the Gaussian curvature.

Using these paths, we now construct an example of $\veps$-collar connections between nearby metric.

\begin{lemma}\label{lemma-collar}
	Let $g_o$ be a metric on $\Sigma$, and $H_o$ be positive function on $\Sigma$, satisfying
	\be \label{eq-convexity-cond}
	R(g_o)-2H_o\Delta(\frac{1}{H_o})-\frac12 H_o^2>0.
	\ee 
	Denote by $\mathcal{G}_\delta$ and $\mathcal{H}_\delta$, the two $\delta$-neighbourhoods 
	
	\bee \mathcal{G}_\delta:=\{ g:\|g-g_o\|_{C^{2,\tau}}<\delta \}\qquad and \qquad \mathcal{H}_\delta:=\{H:\|H-H_o\|_{C^2}<\delta  \}.
	\eee 
	Then for any $\epsilon>0$, there exists a $\delta>0$, such that for any $g_1,g_2\in\Gd$ satisfying $|\S|_{g_1}=|\S|_{g_2}$, and $H_1\in\Hd$ there exists an $\veps$-collar connection from $(\Sigma,g_1,H_1)$ to $g_2$.
\end{lemma}

\begin{proof}
	Fix some small $\veps>0$. We first ask that $\delta$ be sufficiently small so that 
		\bee 
		\inf\limits_{g\in\Gd,H\in\Hd}R(g)-2H\Delta(\frac{1}{H})-\frac12 H^2>0.
		\eee 
		
	We will also require another smallness condition on $\delta$ that is used to make sure our collars have positive scalar curvature. However, in order to state this condition we must introduce some definitions so we reserve explicitly stating this until after equation \eqref{eq-ugly-alpha}.
		
	Now, we fix any two pairs $(g_i,H_i)\in\Gd\times\Hd$, $i=1,2$, with $|\S|_{g_1}=|\S|_{g_2}$. For the chosen $\veps>0$, we seek to construct an $\veps$-collar connection from $(\S,g_1,H_1)$ to $g_2$.
	
	We will make use of the constant
	\bee 
		\chi:=\inf\limits_{\S,g\in\Gd}\{R(g)-2H_1\Delta_g\frac{1}{H_1}\}
	\eee	 
	noting that we have $\chi>\frac12 H_1^2$.
	
	Let $g(s)$ be the path of metrics on $\S$ described above with $g(0)=g_1$ and $g(1)=g_2$, and note that we can ensure that this path stays entirely within $\Gd$. This can be seen by following the construction of such a path using uniformisation in the proofs in \cite{M-S}, of Proposition 1.1 and Lemma 1.2 therein.
	
	We now construct a collar, modelled on part of the spatial Schwarzschild manifold, similar to that used in \cite{MM16} (cf. \cite{CCMM,M-S,M-X}). Recall the spatial Schwarzschild manifold of mass $m$ can be expressed in the form
	\be 
		\gamma_m=dt^2+v_m(t)^2g_*,
	\ee 
	where $g_*$ is the standard round metric on $\S$ of area $4\pi$, and $v_m$ is a smooth positive function satisfying
	\bee 
	v'_m(t)=\sqrt{1-\frac{2m}{v_m(t)}}.
	\eee
	Usually one would parametrise $v(t)$ such that $v(0)=2m$, however for our purposes we set $v(0)=r_o:=\sqrt{\frac{|\Sigma|_{g_o}}{4\pi}}$, the area radius of $g_o$ (and the area radius of both $g_1$ and $g_2$, by construction). Note that this requires us to restrict our attention to $m\leq \frac12 r_o$.

	Using this function $v_m$, we define the metric $\gamma$ on $[0,1]\times \Sigma$ by
	\be \label{eq-metricdefn}
	\gamma:=k^2A(x)^2ds^2+r_o^{-2}v_m(ks)^2g(s),
	\ee 
	where $k$ is a positive constant and $A(x)$ is a positive function on $\Sigma$.
	
	Setting $h(s):=r_o^{-2}v_m(ks)g(s)$, the mean curvature of the level set ${\Sigma_s:=\{s\}\times\Sigma}$ is easily computed as
	\be \label{eq-Hformula}
	H_{(s)}(x)=\frac{1}{2kA(x)}\tr_{h}\dot{h}=\frac{2}{A(x)v_m(ks)}\sqrt{1-\frac{2m}{v_m(ks)}},
	\ee 
	where we make use of the property $\tr_g\dot{g}=0$ here.
	
	In order to prescribe the mean curvature of $\Sigma_0$, this leads us to choose
	\be 
	A(x):=\frac{2}{H_1r_o}\sqrt{1-\frac{2m}{r_o}}.
	\ee
	
	The scalar curvature of $\gamma$ is straightforward to compute as 
	\be \label{eq-Scalarcurv}
	R(\gamma)=r_0^2\( v_m^{-2}\( R(g)-2H_1\Delta_g\frac{1}{H_1}-\frac{1}{2} H_1^2\(1-\frac{2m}{r_o}\)^{-1} \)-\frac{H_1^2}{16}\(1-\frac{2m}{r_o}\)^{-1}k^{-2}|\dot g|_g^2 \),
	\ee 
	which is easily obtained by following Section 1 of \cite{M-S} (cf. \cite{MM16}).
	
	The term
	\bee 
	R(g)-2H_1\Delta_g\frac{1}{H_1}-\frac{1}{2} H_1^2\(1-\frac{2m}{r_o}\)^{-1}
	\eee  in \eqref{eq-Scalarcurv} can be made positive by ensuring $\frac12 H^2_1\(1-\frac{2m}{r_o} \)^{-1}<\chi$, which we do by fixing $m$ appropriately.\\
	
	 To this end, define $\Xi:=\frac{1}{2}\( 1+\frac{H_1^2}{2\chi}\)$, which is strictly less than $1$ by definition of $\chi$. Then we set
	\be \label{eq-mdefn}
	m=\frac{r_o}{2}\( 1-\frac{H_1^2}{2\Xi\chi} \),
	\ee 
	which is positive by assumption. Note that so far there is no requirement that the parameter $m$ be positive, however it will become apparent below that we require this in order to control the mean curvature of the boundary component $\Sigma_1:=\{s=1\}$. In fact, control of the mean curvature with respect to the area growth ($\veps$) is subtle, and is the reason we consider the particular form of the collars here rather than something more generic.
	
	From (\ref{eq-Scalarcurv}) we now have
	\be \label{eq-Scal-2}
	R_\gamma\geq r_o^2\chi\( v_m^{-2}(1-\Xi) -\frac12\Xi k^{-2}\alpha\).
	\ee
	
	So, in particular, if $k^{-2}\alpha$ is sufficiently small then $R_\gamma>0$.

We would like to fix the metric at $\S_1$ to be $(1+\veps)^2 g_2$, which is achieved by choosing $k$ appropriately. In particular, we choose $k$ small enough so that $r_o^{-2}v_m(k)^2=(1+\veps)^2$. Clearly this is possible as the range of $v_m$ is $[r_o,\infty)$. 
	
	Note now that $\sqrt{1-\frac{m}{r_o}}\leq v_m'(t)\leq1$ so we have $r_o+k\sqrt{1-\frac{2m}{r_o}}\leq v_m(k)\leq r_o+k$, so for this choice of $k$ we have $\veps r_o\(1-\frac{2m}{r_o} \)^{-1/2}\geq k\geq \veps r_o$. Plugging in our choice of $m$ gives
\be 
	\veps r_o \frac{\sqrt{2\Xi\chi}}{H_1}\geq k \geq \veps r_o.
\ee 

Substituting this into \eqref{eq-Scal-2} we have

\be 
R_\gamma\geq \chi\( \frac{1-\Xi}{\( 1+\veps H_1^{-1}\sqrt{2\Xi\chi} \)^2} -\frac{\Xi\alpha}{2\veps^2}\).
\ee
That is, scalar curvature is positive if we have
\be \label{eq-ugly-alpha}
	\alpha<\frac{2H_1\veps^2(1-\Xi)}{H_1+\veps\sqrt{2\Xi\chi}}.
\ee 

At this point we remark that as $\delta$ is made small $\alpha$ goes to zero and the right hand side of \eqref{eq-ugly-alpha} approaches
\bee 
	\frac{H_o\veps^2\(1-\frac{H_o^2}{2(R(g_o)-2H_o\Delta_{g_o}H^{-1}_o)}  \)}{H_o+\veps\sqrt{R(g_o)-2H_o\Delta_{g_o}H_o^{-1}+\frac12H_o^2}}>0.
\eee 
In particular, we could arrange that $\delta$ was chosen initially so that \eqref{eq-ugly-alpha} is satisfied for all $g,H\in\Gd\times\Hd$. It follows that the collar manifold has positive scalar curvature.

	Finally, we estimate the mean curvature at the end of the collar using \eqref{eq-Hformula},
	\begin{align}\nonumber
	H_{(s=1)}&=\frac{r_o}{v_m(k)}\( \frac{1-\frac{2m}{v_m(k)}}{1-\frac{2m}{r_o}} \)^{1/2}\,H_1\\
	&=\frac{1}{1+\epsilon}\( \frac{1-\frac{1}{1+\epsilon}\(1- \frac{H_1^2}{2\Xi\chi}\)}{\frac{H_1^2}{2\Xi\chi}} \)^{1/2}H_1\nonumber\\
	&=\frac{1}{1+\epsilon}\( \frac{(1+\epsilon)\frac{2\Xi\chi}{H_1^2}-\(\frac{2\Xi\chi}{H_1^2}- 1\)}{1+\epsilon} \)^{1/2}H_1\nonumber\\
	&=\frac{1}{1+\epsilon}\( \frac{1+\epsilon\frac{2\Xi\chi}{H_1^2}}{1+\epsilon} \)^{1/2}H_1>\frac{1}{1+\epsilon}H_1,\label{eq-H-collar-est}
	\end{align}
	
which completes the proof.

\end{proof}

The condition \eqref{eq-convexity-cond} appears below in the statements of Theorems \ref{thm-collar-glue-in}, \ref{thm-defns} and \ref{thm-cont}. This is entirely due to the fact that these proofs rely on making use of Lemma \ref{lemma-collar} to construct an $\veps$-collar connection. By constructing different $\veps$-collar connections, one could obtain versions of these theorems under different hypotheses. It is not clear if there are optimal circumstances under which such collars can be constructed, so we do not pursue here, a zoo of qualitatively similar collars. We merely remark, that the hypotheses throughout are not optimal in this sense.

Condition \eqref{eq-convexity-cond} is exactly the condition used by Miao and the author \cite{MM16} to attach a horizon to a manifold with boundary, using similar collars. Through an application of the Riemannian Penrose inequality, we obtain a lower bound on the ADM mass of an $n$-manifold in terms of local geometry of a hypersurface $\S$, similar to the lower bound given by the Hawking mass in dimensions 3. In fact, by integrating \eqref{eq-convexity-cond}, one sees that this condition in fact implies positivity of the Hawking mass.

\section{Locally smoothing corners}\label{S-smoothing}
In this section we provide a tool for the proofs to follow, which allows us to smoothly glue together two manifolds along their respective boundaries, while preserving non-negativity of scalar curvature. More specifically, we would like to localise the corner-smoothing technique of \cite{Miao02}. This is done by smoothing the corner and then using the Corvino--Schoen gluing technique \cite{Corvino-00,C-S-06}, in the form generalised by Delay \cite{Delay-11}, to glue the original manifold back on away from the corner. We first recall the definition of an asymptotically flat manifold with corner along a hypersurface.

\begin{definition}[See Miao, \cite{Miao02}]\label{def-corner}
		A manifold with corner along a hypersurface $\Sigma=\p \Omega$ is defined to be a smooth manifold $M$ equipped with a $C^{k,\alpha}$ Riemannian metric $\gamma_-$ on a bounded set $\Omega$ and a $C^{k,\alpha}$ Riemannian metric $\gamma_+$ on $M\setminus\overline \Omega$, such that both metrics are $C^k$ up to the boundary and induce the same metric on $\Sigma$.
		
		We say $M$ is asymptotically flat if $(M\setminus\overline \Omega,\gamma_+)$ is asymptotically flat in the usual sense, and define the ADM mass of $(M,\gamma_-,\gamma_+)$ to be the ADM mass of $(M\setminus\overline \Omega,\gamma_+)$.
\end{definition}

In general, the mean curvature of such a corner $\Sigma$ with respect to $\gamma_-$ and $\gamma_+$ will differ. However, if the mean curvature with respect to $\gamma_+$ is no larger than the mean curvature with respect to $\gamma_-$, then the positive mass theorem still holds \cite{Miao02}. This condition on the mean curvature at the corner was introduced by Bartnik \cite{Bartnik-TsingHua} and an explanation of it is given in Section 2 of \cite{Miao02}. Heuristically, this condition says that the scalar curvature across $\S$ is nonnegative in a distributional sense. The positive mass theorem for manifolds with corner was proven independently by Miao \cite{Miao02}, and Shi and Tam \cite{ShiTam02}.

\begin{thm}[Miao, Shi--Tam \cite{Miao02,ShiTam02}]\label{thm-PMTcorner}
	Let $(M,\gamma_-,\gamma_+)$ be a $C^{2,\tau}$ asymptotically flat manifold with corner along a hypersurface $\S$ and suppose the scalar curvatures of $\gamma_-$ and $\gamma_+$ are nonnegative. Further suppose that 
	\be 
		H_-\geq H_+
	\ee 
	where $H_\pm$ is the mean curvature of $\S$ with respect $g_\pm$ (respectively), and with respect to the normal direction pointing towards infinity. Then the ADM mass of $(M,g\_-,g\_+)$ is nonnegative. Furthermore, if $H_->H_+$ at some point then the mass is strictly positive.
\end{thm}
We remark that a rigidity statement is also obtained in \cite{Miao02}, assuming $(M,\gamma_-,\gamma_+)$ is $C^{3,\tau}$ away from the corner.

Both proofs of Theorem \ref{thm-PMTcorner} are fundamentally different in approach: Shi and Tam's proof is a spinor argument, while Miao's proof involves approximating $(M,\gamma_-,\gamma_+)$ by smooth manifolds with nonnegative scalar curvature and ADM mass converging to that of $(M,\gamma_-,\gamma_+)$. The sequence of approximating manifolds is obtained by locally mollifying the corner, and then by using a conformal change to correct for any negative scalar curvature that may have been introduced by the mollification. Unfortunately this conformal change affects the entire manifold; that is, it is not a local procedure. In order to localise this, we use the Corvino--Schoen gluing technique to glue $(\Omega,\gamma_-)$ and $(M\setminus\overline{\Omega},\gamma_+)$ to the smoothed manifold, away from the corner.

This gluing technique was used by Corvino and Schoen to glue exact Schwarzschild exteriors to general asymptotically flat manifolds. From this they concluded that manifolds with such asymptotics are in fact dense in the space of asymptotically flat manifolds \cite{C-S-06}. It was then noted by Delay \cite{Delay-11} that this gluing technique could be used to glue together much more generic manifolds. In fact, provided that a given manifold is not static where the gluing is to take place, any metric that is sufficiently close (in a $C^k$ sense) can be glued on. Recall, a Riemannian metric is said to be static if the $L^2$-adjoint of the linearised scalar curvature has a non-trivial kernel. That is, if there exists a non-trivial function $f$ satisfying

\be 
	L_{ij}^*(f):=\nabla^2_{ij}f-\Delta(f)g_{ij}-f\Ric_{ij}=0.
\ee 

A well-known result of Beig, Chru\'sciel and Schoen says that generically this operator does indeed have a trivial kernel \cite{BCS-05}. In particular, by perturbing the given metric, we can ensure that no such $f$ exists.

For the readers' convenience, we recall now (and slightly paraphrase) the gluing theorem given in \cite{Delay-11}, which we employ here.

\begin{thm}[Theorem 1.1 of \cite{Delay-11}]\label{thm-Delay}
	Let $V_1\Subset V_2\Subset V_3$ be open proper subsets of each other and of a smooth $n$-dimensional Riemannian manifold $M$, and let $\chi$ be a smooth cut-off function equal to $1$ on $\overline{V_1}$ and vanishing on the complement of $V_2$. Given a metric $\gamma_o$ that is not static, if $\tilde\gamma$ is sufficiently close to $\gamma_o$ in $C^k$, $k\geq \lceil\frac{n}{2}\rceil+6$, on $\overline{V}:=\overline{V_2\setminus V_1}$ then there exists a symmetric covariant 2-tensor $h$ supported on $\overline{V}$, such that the metric
	\be 
		\gamma_s:=\chi \gamma + (1-\chi)\tilde{\gamma}+h
	\ee 
	has scalar curvature interpolating these metrics; i.e.
	\be 
		R(\gamma_s)=\chi R(\gamma)+(1-\chi)R(\tilde \gamma).
	\ee 
\end{thm}
In particular, fixing a metric $\gamma$ that is not static with nonnegative scalar curvature, any sufficiently $C^k$-close metric $\tilde\gamma$ that also has nonnegative scalar curvature can be glued to it in an annular region while preserving the nonnegativity of the scalar curvature.

Combining Miao's corner smoothing technique and the gluing results of Corvino, Schoen and Delay, we obtain the following result on gluing manifolds.

\begin{prop} \label{prop-cornersmooth}
	Let $(M,\gamma_+,\gamma_-)$ be a $C^{k,\tau}$ $3$-manifold with corner along $\Sigma$ as in Definition \ref{def-corner}, with $k\geq8$, and nonnegative scalar curvature.
	Assume that the mean curvature of $\Sigma$ with respect to $\gamma_-$ and $\gamma_+$, denoted $H_-$ and $H_+$ respectively, satisfy
	\be 
		H_+\leq H_-.
	\ee

	In addition, assume that there is a neighbourhood $N$ of $\Sigma$ where $\gamma_-,\gamma_+$ are not static.\\
	
	 Then there exists a $C^k$ metric on $M$ that is identically $\gamma_-$ on $\Omega\setminus N$ and identically $\gamma_+$ on $M\setminus (\Omega\cup N)$, with nonnegative scalar curvature everywhere.
\end{prop}

\begin{proof}
	First recall that the proof of the main theorem in \cite{Miao02} is obtained by constructing a sequence of metrics $\tg_\delta$ that converges to to a given $C^{2}$ metric $g$ in the $C^0$ topology and moreover, converges in $C^2$ away from the corner, $\Sigma$. However, there it is assumed that $(M,\gamma_-,\gamma_+)$ is only of $C^{2,\tau}$ regularity away from $\S$. It is clear from the proof, particularly the proof of Proposition 4.1 therein, that if $(M,\gamma_-,\gamma_+)$ is assumed to be $C^{k,\tau}$ for some $k\geq2$, then the convergence away from $\S$ is with respect to the $C^k$ topology (see also the proof of Theorem 2 therein, on page 1180).
	
	We therefore let $\tg_\delta$ be such an approximating sequence of metrics converging (as $\delta$ goes to zero) to $(\gamma_-,\gamma_+)$ in $C^k$ (away from the corner, with $C^0$ convergence at the corner), which has nonnegative scalar curvature and ADM mass converging to that of $\gamma_+$.

	Let $N_\veps\subset N$ be a small neighbourhood of $\Sigma$ and define $U:=(\Omega\cap N)\setminus N_\veps$. In $U$, we know that $\gamma_-$ is not static, and $\tg_\delta$ converges to $\gamma_-$ in $C^k$. Choosing $\delta$ sufficiently small, we may then directly apply Theorem 1.1 of \cite{Delay-11} to obtain a smooth metric $\widehat g_\delta$ that is equal to $\gamma_-$ on $\Omega\setminus N$ and equal to $\tg_\delta$ on $M\setminus\overline{\Omega}$. Specifically, we set $V_1=\Omega\setminus \overline N$, $V_2=\Omega\setminus \overline{N_\veps}$ and $V_3=\Omega$ in Theorem \ref{thm-Delay}.\\
	
	By shrinking $\delta$ if necessary, the same argument may then be applied to glue $\gamma_+$ in the region $N\setminus \overline{\Omega\cup N_\veps} $. Following the essentially the same argument setting $V_1=M\setminus \overline{N\cup\Omega}$, $V_2=M\setminus \overline{N_\veps\cup\Omega}$ and $V_3=M\setminus\overline{\Omega}$.
\end{proof}

\begin{remark}
	A result of Brendle, Marques and Neves (\cite{BMN11}, Theorem 5) also gives conditions allowing the smoothing of such a corner. In their result, they require strict inequality between the mean curvature on each side of the corner and also may lose a small amount of scalar curvature in a neighbourhood of the corner. However, these limitations are mild and it should not be difficult to prove Proposition \ref{prop-cornersmooth} from their result instead.\footnote{After posting the first version of this article to arXiv, the author was made aware of very closely related work of Jauregui being completed independently \cite{J18}. Therein, Jauregui indeed does use the result of Brendle, Marques and Neves to prove a version of this result.}
\end{remark}

We would like to be able to smooth out the corner in such a way that leaves the entire region $\Omega$ unchanged. In this respect, the above proposition is not quite enough. We would therefore like to smoothly extend a given $\Omega$ by a small amount and then glue on an appropriate exterior region.

\begin{thm}\label{thm-collar-glue-in}
	Let $(M,\gamma_-,\gamma_+)$ be a $C^{k,\tau}$ ($k\geq8$) asymptotically flat $3$-manifold with nonnegative scalar curvature and corner along a hypersurface $\S$, bounding a domain $\Omega\subset M$. Assume the following:
	\begin{enumerate}[(i)]
		\item there exists a smooth manifold $(\Omega_o,\gamma_o)$ with nonnegative scalar curvature in which $(\Omega,\gamma_-)$ embeds isometrically with $\text{\emph{dist}}(\p\Omega,\p \Omega_o)>0$, such that $\gamma_o$ is not static outside $\Omega$;
		\item $\gamma_+$ is not static in some exterior region;
		\item the mean curvature on either side of the corner satisfies
		\bee 
			0<H_+\leq H_-;
		\eee 
		\item the Bartnik data $(\p\Omega,g_\S,H_-)$, where $g_\S$ is the restriction of $\gamma_\pm$ to $\p\Omega$, satisfies
		\bee \label{eq-convexity-cond2}
		R(g_\S)-2H_-\Delta_{g_\S}(\frac{1}{H_-})-\frac12 H_-^2>0.
		\eee 
	\end{enumerate} 
	
			Then, there exists a complete asymptotically flat manifold $(\widehat M,\widehat{\gamma})$ with nonnegative scalar curvature such that $\widehat M$ minus a compact set contains two disjoint component; one isometric to $(M\setminus K,\gamma_+)$, where $K$ is some other compact set containing $\overline{\Omega}$, and the other piece isometric to $(\Omega,\gamma_-)$.
\end{thm}
Intuitively, one can view this as a way to slightly nudge the corner away from $\Omega$ before smoothing it.

\begin{proof}
	Let $(\Omega_o,\gamma_o)$ be a manifold with nonnegative scalar curvature, containing $(\Omega,\gamma_-)$ as described above, and let $(\Omega_\delta,\gamma_\delta)$ be the manifold obtained by intersection $\Omega_o$ with a $\delta$-neighbourhood of $\Omega$. We will be interested in choosing $\delta$ small, so as to only extend $\Omega$ by a small amount. Denote by $g_1$ the induced metric on $\p \Omega_\delta$, and denote by $H_1$ the mean curvature of $\p\Omega_\delta$. We would like to ensure that $g_1$ is close to $g_\S$ in $C^{2,\tau}$, which can be done by taking $\delta$ sufficiently small. Similarly, we can also ensure that $H_1$ is close to $H_-$ in $C^2$.
	
	We aim to use condition (ii) to construct an $\veps$-collar connection from $g_1$ to $g_\S$. We therefore define $g_2:=\zeta g_\S$, where $\zeta:=|\S|_{g_1}|\S|_{g_\S}^{-1}$, so that the area of $\S$ with respect to $g_1$ is the same as that with respect to $g_2$. Note that $\zeta$ approaches $1$ as $\delta$ approaches zero.
	
	We fix some $\veps>0$, and in what follows we will continually shrink $\delta$ to ensure certain conditions are met. First, we ensure $\delta$ is small enough to apply Lemma \ref{lemma-collar} to construct an $\veps$-collar connection $\C$ from $(\S,g_1,H_1)$ to $g_2$. The Bartnik data on one boundary component of $\C$ is exactly $(\S,g_1,H_1)$ while the boundary data on the other boundary component is $(\S,(1+\veps)^2g_2,\widehat{H})$, where $\widehat{H}>(1+\veps)^{-1}H_1$. In particular, we have
	\bee 
		\frac{1}{1+\veps}H_+\leq \frac{1}{1+\veps}H_-<\widehat{H}+\frac{H_--H_1}{1+\veps}.
	\eee 
	By shrinking $\delta$ further, we can therefore ensure that we have 
	\be \label{eq-Hest}
	\widehat{H}>\frac{1}{1+\varepsilon}H_-.
	\ee 	
	
	In order to glue this collar to the exterior manifold $N:=M\setminus\overline{\Omega}$, consider the rescaled exterior metric
	\bee 
		\gamma_E:=\zeta\(1+\veps\)^2\gamma_+,
	\eee
	which has induced boundary metric $(1+\veps)^2g_2$ and mean curvature $H_E=\zeta^{-1/2}(1+\veps)^{-1} H_+$. By \eqref{eq-Hest} we then have
	\bee 
				H_E<\zeta^{-1/2}\widehat{H},
	\eee 
	so by further shrinking $\delta$ if necessary, we can ensure $H_E<\widehat{H}$, as required for the gluing.
	
	That is, we have shown that for any $\veps>0$, we can choose $\delta>0$ sufficiently small to ensure the existence of a collar manifold with two boundary components; the first inducing Bartnik data $(\S,g_1,H_1)$ and the second inducing the Bartnik data $(\S,\zeta(1+\veps)^2 g_\S,\widehat{H})$, where $\widehat{H}>H_E$. It should be kept in mind that $g_1$ and $H_1$ depend on $\delta$, and approach $g_\S$ and $H_-$ respectively, as $\delta$ goes to zero.
	
	Now, we may apply Proposition \ref{prop-cornersmooth} to obtain a smooth Riemannian manifold $(\widehat{M},\widehat{\gamma})$ containing one subset that is isometric to $(\Omega,\gamma_-)$, and another subset that is isometric to an exterior region in $M$ with the metric $\zeta(1+\veps)^2\gamma_+$.
	
	To complete the proof, we simply note that we can choose $\veps$ sufficiently small -- and in turn, $\delta$ small -- to ensure that $\zeta(1+\veps)^2\gamma_+$ is $C^k$-close enough to $\gamma_+$ to apply Theorem \ref{thm-Delay} and glue $\gamma_+$ to the exterior.
	
\end{proof}
We briefly explain the hypotheses (i)--(iv) in Theorem \ref{thm-collar-glue-in}. Condition (i) is rather weak, and in the context of the Bartnik mass it is entirely natural, as one would generally take $\Omega$ to be a bounded subset of a Riemannian manifold with nonnegative scalar curvature. Condition (ii) is required for standard gluing techniques. Condition (iii) is the usual mean curvature condition for a manifold with corners. Finally, condition (iv) is simply enforced in order to ensure the existence of an $\veps$-collar connection. As mentioned at the end of the preceding section, it should be possible to replace this last condition by constructing different kinds of $\veps$-collar connections.

\section{Equivalence of definitions} \label{S-defns}
As mentioned in Section \ref{S-intro}, the term ``Bartnik mass" often is used in the literature to indicate slightly different, conjecturally equivalent, quantities. In this section, we use the tools developed in Section \ref{S-collar} to give conditions under which some of these definitions are indeed equivalent.

Let $(\Omega,\gamma)$ be a compact manifold with boundary and nonnegative scalar curvature. One of the points of discrepancy in the definitions of the Bartnik mass is in choosing a non-degeneracy condition. The two most common non-degeneracy conditions in the literature, on an extension $(M,\widehat\gamma)$ of $(\Omega,\gamma)$ are:
\begin{enumerate}[(I)]
	\item there exists no close minimal surfaces in $M$ containing $\Omega$,
	\item the boundary $\p M$ is outer minimising in $M$.
\end{enumerate}

In what follows, when we speak of the non-degeneracy condition, we will be referring to (II) above. It is not at all obvious how to ensure condition (I) is preserved in constructions like those considered here, and in fact Jauregui gives a nice discussion on the problems that could arise in \cite{J18}.

Another point of discrepancy between definitions of the Bartnik mass is in essence, whether or not the admissible extensions are permitted to have a corner at $\p\Omega$. We define three spaces $\mathcal{A}_i$ of manifolds that are sometimes called admissible extensions of $(\Omega,\gamma)$.\\

\begin{enumerate}[1.]
	\item 
 Define the set $\mathcal{A}_1(\Omega,\gamma)$ as the set of smooth asymptotically flat manifolds, with nonnegative scalar curvature, in which $(\Omega,\gamma)$ embeds isometrically, satisfying the non-degeneracy condition.\\
 
\item 
Define the set $\mathcal{A}_2(\Omega,\gamma)$ as the set of smooth asymptotically flat manifolds with a corner along a hypersurface $\Sigma$, with nonnegative scalar curvature, such that $\Sigma$ bounds a domain isometric to $(\Omega,\gamma)$, with mean curvature on either side of the corner satisfying $H_-=H_+$, and satisfying the non-degeneracy condition.\\

\item 
Define the set $\mathcal{A}_3(\Omega,\gamma)$ as the set of smooth asymptotically flat manifolds with a corner along a hypersurface $\Sigma$, with nonnegative scalar curvature, such that $\Sigma$ bounds a domain isometric to $(\Omega,\gamma)$, with mean curvature on either side of the corner satisfying $H_-\geq H_+$, and satisfying the non-degeneracy condition.
\end{enumerate}

We now define three notions of Bartnik mass, associated to each of these definitions of admissible extensions that appear in the literature. That is, by a slight abuse of notation, we define
\be 
	\m_i(\Omega,\gamma):=\inf\{\m_{ADM}(M): M\in\mathcal{A}_i \}.
\ee

\begin{remark}
	Rather than explicitly considering manifolds with corner in the definitions $\mathcal{A}_2$ and $\mathcal{A}_3$, one usually considers manifolds with boundary, whose boundary induces Bartnik data corresponding to that of $\p\Omega$. This allows one to write, the Bartnik mass as a functional on the space of Bartnik data rather than on a space of compact $3$-manifolds with boundary. That is, in the case of the $\m_2$ and $\m_3$, often we speak of the Bartnik mass of given Bartnik data $(\S,g,H)$, denoted $\m_i(\S,g,H)$, rather than the Bartnik mass of a compact manifold with boundary.
\end{remark}

	We remark that the sets of extensions considered in the three definitions are subsets of each other. In particular, it is clear that we have
	\be 
	\m_1\geq \m_2\geq \m_3.
	\ee 
	Theorem \ref{thm-defns} below gives conditions under which we can say that these three definitions are in fact equal.

\begin{thm}\label{thm-defns}	Let $(\Omega,\gamma)$ be a compact $3$-manifold with nonnegative scalar curvature and boundary $\p\Omega=\S$. Assume the following:
	\begin{enumerate}[(i)]
		\item there exists a smooth manifold $(\Omega_o,\gamma_o)$ with nonnegative scalar curvature in which $(\Omega,\gamma)$ embeds isometrically with $\text{\emph{dist}}(\p\Omega,\p \Omega_o)>0$, such that $\gamma_o$ is not static outside $\Omega$ and $\p\Omega$ is strictly outer-minimising;
		\item the Bartnik data $(\S,g,H)$, where $g$ is the restriction of $\gamma$ to $\S=\p\Omega$, satisfies
		\bee 
		R(g)-2H\Delta_g(\frac{1}{H})-\frac12 H^2>0;
		\eee 
	\end{enumerate} 
	
	Then
	\bee
		\m_1(\Omega,\gamma)=\m_2(\p\Omega,g,H)=\m_3(\p\Omega,g,H).
	\eee 
\end{thm}
\begin{proof}
	We simply must prove $\m_1\leq \m_3$. This follows from Theorem \ref{thm-collar-glue-in}.

	For any small $\mu>0$, we can construct an extension $(M,\gamma_\mu)\in\mathcal{A}_3(\Omega,\gamma)$ with ADM mass satisfying $\m_3\leq\m_{ADM}(M,\gamma_\mu)<\m_3+\mu$. Furthermore, it is clear we can arrange this so as to not be static in the exterior. From Theorem \ref{thm-collar-glue-in} we can then obtain $(\widehat M,\widehat \gamma_\mu)\in\mathcal{A}_1(\Omega,\gamma)$ with ADM mass $m_{ADM}(\widehat{M},\widehat\gamma_\mu)=m_{ADM}(M,\gamma_\mu)<\m_3+\mu$.
	
	 That is, provided the non-degeneracy condition is preserved, we have $\m_1<\m_3+\mu$, and taking $\mu$ to be arbitrarily small would complete the proof.
		
	However, it is a subtle point to check that this procedure can indeed be done without violating the non-degeneracy condition. Note that the proof of Theorem \ref{thm-collar-glue-in} involves performing a gluing away from $\p\Omega$; that is, somewhere in the local extension $(\Omega_o,\gamma_o)$. Since $\p\Omega$ is strictly mean convex, then there exists a $\widehat\S$ containing $\p\Omega$ that is also strictly mean convex with area strictly larger than that of $\p\Omega$. Set $\delta=|\widehat\S|-|\p\Omega|>0$, and we consider such a $\widehat{\S}$ to be the boundary of $\p\Omega_\delta$ where the $\veps$-collar connection is glued to in the proof of Theorem \ref{thm-collar-glue-in}. By construction, $\widehat\S$ is also outer-minimising in the manifold with corners that we obtain by attaching an $\veps$-collar connection to $\widehat\S$ and connecting it to $(M,\gamma_\mu)$. This is because the collars we construct are foliated by strictly convex surfaces and the asymptotically flat extension glued to the end of the collar has outer-minimising boundary. Recall now that the smoothing and gluing procedure results in a sequence of smooth metrics that are isometric to $\gamma_o$ near $\p\Omega$, converging to our manifold with corners in $C^0$. Therefore, it can be arranged that any surface $\S'$ containing $\widehat{\S}$ has area at least $|\widehat{\S}|-\frac{\delta}{2}>|\p\Omega|$. To show that $\p\Omega$ is outer-minimising in this smoothed extension, it only remains to show that there are no closed surfaces surrounding $\p\Omega$, crossing $\widehat{\S}$, that have area less than that of $\p\Omega$. However, by the same reasons as above, applied to the portion exterior to $\widehat{\S}$, of any surface surrounding $\p\Omega$, one can see that no such surface can exist. We therefore conclude that $\p\Omega$ is strictly outer-minimising in the smoothed manifold, provided that we are sufficiently close in $C^0$ to the manifold with corners that we constructed.

\end{proof}
It should be remarked here that the preservation of the non-degeneracy condition could be included directly in the statement of Theorem \ref{thm-collar-glue-in}. However, we elect to present this here for the sake of exposition. Furthermore, the situation considered here is covered by a more general result on preservation of this non-degeneracy condition obtained very recently by Jauregui (see Lemma 13 of \cite{J18}).

\begin{remark} \label{rmk-CMC}
	Given Bartnik data $(\S,g,H_1)$ and $(\S,g,H_2)$ with $0<H_2\leq H_1$, it is clear that $\mathcal{A}_3(\S,g,H_2)\subset\mathcal{A}_3(\S,g,H_1)$. It is therefore clear that $\m_3(\S,g,H_1)\leq \m_3(\S,g,H_2)$, and therefore under the hypotheses of Theorem \ref{thm-defns} the same can be said about the other definitions of Bartnik mass.
	
	In particular, under these hypotheses, we have
	\bee 
		\m_B(\S,g,H)\leq \m_B(\S,g,\min\limits_\S(H)),
	\eee
	and therefore any estimate of the Bartnik mass of a constant mean curvature (CMC) surface gives rise to estimates for non-CMC surfaces too. For example, the estimates obtained by Lin and Sormani \cite{L-S}, and those by Cabrera Pacheco, Cederbaum, Miao and the author \cite{CCMM}, can be used to estimate non-CMC Bartnik data, provided it satisfies the hypotheses of Theorem \ref{thm-defns}. Indeed one could repeat the argument given above to glue CMC extensions to non-CMC data directly.
\end{remark}

Note that the hypotheses of Theorem \ref{thm-defns} are not very restrictive in general. The following proposition gives conditions ensuring that an appropriate local extension exists.
\begin{prop}\label{prop-nonstat}
	Let $(\Omega,\gamma)$ be a smooth compact manifold with nonnegative scalar curvature and smooth strictly mean convex boundary $\p\Omega$. Assume there exists a smooth manifold $(\Omega_o,\gamma_o)$ in which $(\Omega,\gamma)$ isometrically embeds with $\text{\emph{dist}}(\p\Omega,\p \Omega_o)>0$ and nonnegative scalar curvature.
	
	Then there exists a $3$-manifold $(\widehat\Omega,\widehat\gamma)$ in which $(\Omega,\gamma)$ isometrically embeds with $\text{\emph{dist}}(\p\Omega,\p \Omega_o)>0$, satisfying:
	\begin{itemize}
		\item outside of $\Omega$, $\widehat{\gamma}$ has positive scalar curvature;
		\item $\widehat{\gamma}$ is not static outside of $\Omega$;
		\item $\p\Omega$ is strictly outer-minimising in $\widehat{\Omega}$.
	\end{itemize}
\end{prop}
\begin{proof}
	In a neighbourhood of $\p\Omega$, the metric $\gamma_o$ can be expressed as
	\be 
		\gamma_o=ds^2+g(t),
	\ee 
	for $t\in[0,\veps)$, with $t=0$ corresponding to $\p\Omega$.

	The scalar curvature of $\gamma_o$ is then computed in this neighbourhood as
	\be 
		R(\gamma_o)=R(g(t))-\tr_{g(t)}\ddot{g}(t)-\frac14(\tr_{g(t)}g(t))^2+\frac34|\dot{g(t)}|_{g(t)}^2.
	\ee 
The proof is now essentially the proof of Lemma 2.3 in \cite{M-S}, where part of the Schwarzschild manifold was deformed to have strictly positive scalar curvature. We consider the deformed metric
	\be 
	\gamma_\sigma=ds^2+g(\sigma(t)),
	\ee 
where $\sigma$ is some increasing function. We then compute the scalar curvature of this new metric to obtain
	\be 
	R(\gamma_\sigma)=R(g(\sigma))-(\dot{\sigma})^2R(g(\sigma))-\ddot{\sigma}\tr_{g(\sigma)}\dot{g(\sigma)}+(\dot{\sigma})^2R(\gamma_o),
	\ee
where we write $\sigma$ to mean $\sigma(t)$. We now choose 
\bee 
\sigma(t)=t-\int_0^t\exp\( -s^{-2} \)ds,
\eee 
which gives
\begin{align*} 
	R(\gamma_\sigma)&=\exp(-t^{-2})\(R(g(\sigma))+2t^{-3}\tr_{g(\sigma)}\dot{g(\sigma)}\)+(\dot{\sigma})^2R(\gamma_o)\\
	&=\exp(-t^{-2})\(R(g(\sigma))-2R(\gamma_o)+\exp(-t^{-2})R(\gamma_o)+2t^{-3}\tr_{g(\sigma)}\dot{g(\sigma)}\)+R(\gamma_o).
\end{align*} 
Since $\dot{\sigma}(0)=1$ and all higher derivatives vanish at $t=0$, $\gamma_\sigma$ does indeed smoothly extend $\Omega$. Since $R(\gamma_o)>0$ is non-negative by assumption, clearly $R(\gamma_\sigma)>0$ for some neighbourhood of $\Omega_o$. Furthermore, it is now well-known (see \cite{Corvino-00}) that a static metric must have constant scalar curvature, so we can conclude this extension is not static. 

Finally, since $\p\Omega$ is mean convex, it is clear that all small outward perturbations of the boundary increase area so for a sufficiently small neighbourhood $\p\Omega$ is strictly outer-minimising.
\end{proof}

The question of when a manifold with boundary is extendible in this sense, to a larger manifold with non-negative scalar curvature, is non-trivial. It seems reasonable to conjecture that this is usually possible, however if at the boundary the scalar curvature is vanishing and strictly decreasing in the outward direction, then this clearly cannot be possible. However, Reiris has some interesting results demonstrating that a smooth manifold with strictly mean convex and compact boundary can be $C^2$ extended in such a way that preserves non-negativity of scalar curvature \cite{Reiris}. For the purposes here, we only require that the manifold be sufficiently smooth to perform the gluing away from the boundary. It is therefore likely that one could use the results of Reiris carefully to ensure a local extension suitable for our purposes always exists, however this is beyond the scope of the current work. We simply would like to emphasise that the local extendibility of $\Omega$ is a very weak assumption in our context. Indeed the Bartnik mass is generally defined for a subset of a Riemannian manifold with nonnegative scalar curvature so for the physically interesting case, the point is effectively moot.

\section{Continuity of the Bartnik mass}\label{S-Cont}
In this final section, we now turn to address the continuity of the Bartnik mass with respect to its Bartnik data. Throughout this section, we will focus on the case where all three definitions of Bartnik mass agree. For simplicity, we will therefore take the Bartnik mass to be that given by $\m_2$ in the preceding section and denote it simply by $\m_B$. As above, the key to the proof is the existence of an $\veps$-collar connections, which allow us to connect given Bartnik data to an admissible extension of nearby Bartnik data. Before stating the theorem, we recall a well-known fact about how the Bartnik mass behaves under scaling of the data.

Given a constant $\lambda>0$ the Bartnik mass satisfies
\be 
	\m_B(\S,\lambda^2g,\frac{1}{\lambda}H)=\lambda \m_B(\S,g,H).
\ee

This can be observed simply by noting that if $(M,\gamma)$ is an admissible extensions of the data $(\S,g,H)$ with mass $m$, then $(M,\lambda^2\gamma)$ is an admissible extension of the data $(\S,\lambda^2g,\frac{1}{\lambda}H)$ with mass $\lambda m$, which is seen directly from \eqref{eq-ADMdefn}.

In order to prove the main result of this section, we will first need to show that the Bartnik mass can be uniformly bounded in a small neighbourhood.
	
\begin{lemma} \label{lem-unibound}
	Let $(\Sigma,g_o,H_o)$ be given Bartnik data satisfying
	\be \label{eq-deltaHcond}
		R(g_o)-2\Delta_{g_o}\frac1H>0,\qquad\text{and}\qquad H>0.
	\ee 
	 Then there is a neighbourhood of $(g_o,H_o)$ in $C^{2,\tau}\times C^2$ and a constant $C$ such that the Bartnik mass on this neighbourhood is less than $C$.
\end{lemma}
\begin{proof}
	Take the neighbourhood to be small enough so that \eqref{eq-deltaHcond} holds for each $(g,H)$ in the neighbourhood. Now let $g(s)$ be the path constructed in Section \ref{S-collar}, with $g(0)=g$ and $g(1)$ being a round metric of the same area. Clearly in this case, $\alpha$ will not be small generically. Nevertheless, we define the metric $\gamma$ as in \eqref{eq-metricdefn} and recall the scalar curvature expression from \eqref{eq-Scalarcurv}:
	\bee
R(\gamma)=r_0^2\( v_m^{-2}\( R(g)-2H\Delta_g\frac{1}{H}-\frac{1}{2} H^2\(1-\frac{2m}{r_o}\)^{-1} \)-\frac{H^2}{16}\(1-\frac{2m}{r_o}\)^{-1}k^{-2}|\dot g|_g^2 \),
\eee 

where we again choose $A(x)$ as in Section \ref{S-collar} prescribe the mean curvature at $t=0$ to be $H$. The term
\bee 
R(g)-2H\Delta_g\frac{1}{H}-\frac{1}{2} H_1^2\(1-\frac{2m}{r_o}\)^{-1}
\eee
can again be made positive by choosing $m$ appropriately; in fact, for our purposes here we can take $m<<0$ so as to ensure, for whatever $k>0$ we choose, $R(\gamma)>0$.

That is, we can construct a collar manifold that realises the correct boundary data on one connected component of the boundary and is round with positive mean curvature at the other component. We can write the metric at the end of the collar as $g_1=r_1^2g_*$, where $r_1$ is positive constant (the area radius) and $g_*$ is the standard round metric of area $4\pi$. Let $\gamma_{r_1}$ be the Schwarzschild manifold of mass $r_1/2$ and note that the horizon of this manifold is isometric to $g_1$. By identifying the horizon with the end of our collar, we obtain a manifold with corner and if Schwarzschild was not static we could immediately apply Proposition \ref{prop-cornersmooth} to obtain an admissible extension with mass $r_1/2$. Note that, as in the proof of Theorem \ref{thm-defns}, this can be done so that the non-degeneracy condition is preserved.

This problem of the Schwarzschild manifold being static is easily circumvented by considering a Schwarzschild manifold with slightly larger mass and then bending it slightly near the horizon. This can even be done explicitly as in Lemma 2.3 of \cite{M-S} (or by applying Proposition \ref{prop-nonstat} appropriately). Clearly the area at the end of such a collar is controlled by the Bartnik data and therefore we have an upper bound for the Bartnik mass for all Bartnik data in the neighbourhood.

\end{proof}

We now are in a position to prove the main result of this section.
\begin{thm}\label{thm-cont}
Let $(\Sigma,g_o,H_o)$ be given Bartnik data such that 
	\be \label{eq-conthypothesis}
	R(g_o)-2H_o\Delta_{g_o}(\frac{1}{H_o})-\frac12 H_o^2>0.
	\ee 
If the non-degeneracy condition is taken to be (I), then further assume that the Bartnik mass satisfies $\m(\S,g_o,H_o)\leq \( \frac{|\S|}{16\pi} \)^{1/2}$.

Then the Bartnik mass is continuous at $(g_o,H_o)$ with respect to the $C^{2,\tau}\times C^2$ topology.
\end{thm}
\begin{proof}
	In what follows, the non-degeneracy condition is preserved exactly as in the proof of Theorem \ref{thm-defns}. We omit repeating the details for the sake of exposition.
	
	Let $\mathcal{G}_\delta$ and $\mathcal{H}_\delta$ denote $\delta$-neighbourhoods of $(g_o,H_o)$ as in Lemma \ref{lemma-collar}.
	Fix some small $\mu>0$. We begin by choosing some $\veps>0$ small, which will be the parameter in our $\veps$-collar connections. Specifically, we choose $\veps<\frac{\mu}{3}\m_B(\S,g_o,H_o)^{-1}$.
	
	 Before proceeding, we must more closely look at the mean curvature $\widehat{H}$ of the end of the collar constructed in Lemma \ref{lemma-collar}. Given some Bartnik data $(\S,g,H)$, a collar is constructed with mean curvature $\widehat{H}$ at one boundary component, satisfying $\widehat{H}>\frac{1}{1+\veps}H$. In what follows, we would like to work uniformly in $H$, so we would like to have some control on the difference $\widehat{H}-\frac{1}{1+\veps}H$. From \eqref{eq-H-collar-est}, it can be seen that, fixing some $\veps,\delta_1>0$, for all $(g,H)\in\Gd\times\Hd$ there is a constant $C(\delta_1,\veps)$ such that the collars constructed by Lemma \ref{lemma-collar} satisfy
	\be 
		\widehat{H}-\frac{1}{1+\veps}H\geq C(\delta_1,\veps).
	\ee 
	Once this a priori $\delta_1$ has been chosen, we will restrict our choice of $\delta$ in what follows to satisfy
	\be 
		\delta\leq (1+\veps)C(\delta_1,\veps).
	\ee 
	This implies that for such $\widehat H$ associated to a given $(g,H)\in\Gd\times\Hd$ via our $\veps$-collar construction, any other $H_2\in\Hd$ satisfies $\widehat{H}>\frac{1}{1+\veps}H_2$.\\
	
	Recall, we seek to find $\delta>0$ such that for all $(g,H)\in\Gd\times\Hd$,
	\bee 
		|\m_B(\S,g,H)-\m_B(\S,g_o,H_o)|<\mu.
	\eee

	Pick some nearby $H\in\Hd$ and $g\in\Gd$ and we write $g_1=\zeta^2 g$, $\zeta^2:=|\S|_{g_o}|\S|_{g}^{-1}$, so that $g_1$ has the same area as $g_o$. For the purposes of later rescaling, we also define $H_1=\zeta^{-1}H$.
	
	Fix $\delta$ small enough so that to ensure that Lemma \ref{lemma-collar} applies to each such $(g_1,H_1)$ as defined above, then shrink $\delta$ as necessary to ensure we have the uniform control on the mean curvature at the end of the constructed collars, as described above. Note that this $\delta$ depends only on $g_o,H_o$ and $\mu$ (via $\veps$).
	
	This gives us a collar manifold $(\C\cong[0,1]\times\S,\gamma_{\C})$ that induces Bartnik data $(\S,g_1,H_1)$ and $(\S,(1+\veps)^2g_o,\widehat{H})$ on the boundary components, where $\widehat{H}>\frac{1}{1+\veps}H_o$.
	
	Now, we can construct an extension $(M,\gamma_{\veps,\mu})$ in $\mathcal{A}_2(\S,(1+\veps)^2g_o,\frac1{1+\veps}H_o)$ with ADM mass
	\begin{align*} 
		\m_{ADM}(M,\gamma_{\veps,\mu})&\leq \m_B(\S,(1+\veps)^2g_o,\frac{1}{(1+\veps)}H_o)+\frac{\mu}{3}\\
		&=(1+\veps)\m_B(\S,g_o,H_o)+\frac{\mu}{3}<\m_B(\S,g_o,H_o)+\frac{2\mu}{3},
	\end{align*}
	where we recall the original condition on $\veps$.
	
	By construction, we can attach this extension to the $\veps$-collar connection given above, to obtain a manifold with corner having the correct boundary behaviour for the positive mass theorem with corners. In particular, we can apply Theorem \ref{thm-collar-glue-in} to obtain an admissible extension $(M,\widehat\gamma_{\mu})$ in $\mathcal{A}_2\( \S,g_1,H_1 \)$ with mass bound from above by $\m_B(\S,g_o,H_o)+\frac{2\mu}{3}$.
	
	All that remains is to scale this extension by $\zeta$ to obtain
	\be 
		\m_B(\S,g,H)=\zeta^{-1}\m_B(\S,g_1, H_1)<\zeta^{-1}\( \m_B(\S,g_o,H_o)+\frac{2\mu}{3} \).
	\ee 
	Now, recall that $\zeta$ can be made arbitrarily close to $1$ by shrinking $\delta$. So we restrict our choices of $\delta$, still depending only on $g_o,H_o,\mu$, to ensure
	
		\be 
		\m_B(\S,g,H)<\m_B(\S,g_o,H_o)+\mu,
		\ee 

for all $(g,H)\in\Gd\times\Hd$. That is, we have established upper semi-continuity of $\m_B$.\\

Now we simply must reverse the argument.\\

We can construct a $\veps$-collar connection from $(\S,g_o,H_o)$ to $g_1$, such that the Bartnik data on the other end of the collar is $\(\S,(1+\veps)^2g_1,\widehat{H}_1\)$ with $\widehat{H}_1>\frac{1}{1+\veps}H_1$. Now we take an admissible extension $(M,\tilde\gamma_{\veps,\mu})\in\mathcal{A}_2(\S,(1+\veps)^2g_1,\frac{1}{1+\veps}H_1)$ with ADM mass
\begin{align*} 
	\m_{ADM}(M,\tilde{\gamma}_{\veps,\mu})&<(1+\veps)\m_B(\S,g_1,H_1)+\frac{\mu}{3}\\
	&=(1+\veps)\zeta\m_B(\S,g,H)+\frac{\mu}{3}\\
	&<\m_B(\S,g,H)+O(\veps+|1-\zeta|)\sup\limits_{\Gd\times\Hd}\m_B+\frac{\mu}{3},
\end{align*}
where the supremum is the supremum of the Bartnik mass over the neighbourhood in consideration. Provided $\delta$ is sufficiently small, Lemma \ref{lem-unibound} implies this supremum is bound. Therefore we can take $\delta$ and $\veps$ to be small enough to ensure we have
\be 
	\m_{ADM}(M,\tilde{\gamma}_{\veps,\mu})<\m_B(\S,g,H)+\mu.
\ee

In particular, by again applying Theorem \ref{thm-collar-glue-in}, we obtain admissible extensions of $(\S,g_o,H_o)$ and obtain
\be 
	\m_B(\S,g_o,H_o)<\m_B(\S,g,H)+\mu,
\ee 
completing the proof.
\end{proof}


\vspace{7mm}

\begin{thebibliography}{10}
	
	\bibitem{ADM} Arnowitt, R.; Deser, S., and Misner, C. W., {\sl Coordinate invariance and energy expressions in general relativity}, Phys. Rev., {\bf 122} (1961), no. 3, 997--1006.
	
	\bibitem{Bartnik-86} Bartnik, R., {\sl The mass of an asymptotically flat manifold}, Comm. Pure Appl. Math.  \textbf{39} (1986), no. 5, 661--693.
	
	\bibitem{A-J} Anderson, M.; Jauregui, J., {\sl Embeddings, immersions and the Bartnik quasi-local mass conjectures}, arXiv preprint:  	arXiv:1611.08755.
	
	
	\bibitem{Bartnik-89} Bartnik, R., {\sl A new definition of quasilocal mass}, Phys. Rev. Lett.  \textbf{62} (1989), 2346.
	
	\bibitem{Bartnik-TsingHua} Bartnik, R., {\sl Energy in General Relativity}, Tsing Hua lectures on Geometry and Analysis (Hsinchu, 1990--1991), International Press, Cambridge MA (1997), 5--27.
	
	\bibitem{BCS-05} Beig, R.; Chru\'siel, P.; Schoen, R.
	{\sl KIDS are non-generic}, 
	Ann. Henri Poicar\'e, {\bf 6} (2005), 155--194.
	
	
	\bibitem{BMN11} Brendle, S.; Marques, F. C.; Neves, A.,
	{\sl Deformations of the hemisphere that increase scalar curvature}, 
	Invent. math, {\bf 185}, no. 1 (2011), 175--197.

	\bibitem{CCM} {Cabrera Pacheco}, A. J.; Cederbaum, C.; McCormick, S.,
	{\sl Asymptotically hyperbolic extensions and an analogue of the Bartnik mass}, preprint, arXiv:1802.03331.
	
	\bibitem{CCMM} {Cabrera Pacheco}, A. J.; Cederbaum, C.; McCormick, S.; Miao, P.,
	{\sl Asymptotically flat extensions of CMC Bartnik data}, 
	Class. Quantum Grav., {\bf 34} (2017), no. 10, 105001
	
	\bibitem{Chrusciel}
		Chru\'sciel,  P.,  {\sl Boundary conditions at spatial   infinity from a Hamiltonian point of view},
		Topological Properties and Global Structure of Space-Time, Plenum Press, New York, (1986), 49--59.
		
	
	\bibitem{Corvino-00} Corvino, J., {\sl Scalar curvature deformation and a gluing construction for the Einstein constraint equations}, Commun. Math. Phys.  \textbf{214}, (2000), 137--189.
	
	\bibitem{C-S-06} Corvino, J.; Schoen, R., {\sl On the asymptotics for the vacuum Einstein constraint 	equations}, J. Diff. Geom.  \textbf{73}, (2006), 185--217.
	
	\bibitem{Delay-11} Delay, E., {\sl Localized gluing of Riemannian metrics in interpolating their scalar curvature}, Diff. Geom. App.  \textbf{29}, 3, (2011), 433--439.
	
	
	\bibitem{H-I01} Huisken, G.; Ilmanen, T., {\sl The inverse mean curvature flow and the {R}iemannian {P}enrose inequality},
	J. Differential  Geom., {\bf 59} (2001), no. 3, 35$3$--437.
	
	\bibitem{J18}  Jauregui, J., 
		{\sl Smoothing the Bartnik boundary conditions and
			other results on Bartnik's quasi-local mass},  preprint, arXiv:1806.08348.
	
	\bibitem{L-S} Lin, C-Y.;  Sormani, C., {\sl Bartnik's mass and Hamilton's modified Ricci flow}, Ann. Henri Poincar\'e, {\bf 17} (2016), no. 10, 2783--2800.
	
	\bibitem{M-S} Mantoulidis, C.;  Schoen,  R., {\sl On the Bartnik mass of apparent horizons}, Class. Quantum Grav., {\bf 32} (2015), no. 20, 205002, 16pp.
	\bibitem{MM16} McCormick, S.; Miao, P.,  
	{\sl On a Penrose-like inequality in dimensions less than eight}, Int. Math. Res. Not., rnx181 (2017).
	\bibitem{Miao02}
	Miao, P., {\sl Positive mass theorem on manifolds admitting corners along a hypersurface}, Adv. Theor. Math. Phys., {\bf 6} (2002), no. 6, 116$3$--1182.
	
	\bibitem{M-X}  Miao, P.;  Xie, N.-Q., 
	{\sl On compact $3$-manifolds with nonnegative scalar curvature with a CMC boundary component},  preprint, arXiv:1610.07513.
	
	\bibitem{Reiris}  Reiris, M.,
		{\sl A note on scalar curvature and the convexity of boundaries},  preprint, arXiv:1209.4525.
	
	\bibitem{ShiTam02} Shi, Y.-G.;  Tam,  L.-F.,
	{\sl Positive mass theorem and the boundary behaviors of compact manifolds with nonnegative scalar curvature},
	J. Differential Geom. \textbf{62}  (2002),  no.1, 79--125.
	
	
	
\end{thebibliography}
\end{document}